\documentclass[preprint,12pt]{elsarticle}

\usepackage{amssymb}

\usepackage{amsthm}
\usepackage{amsmath}

\usepackage[matrix,arrow,curve,frame,arc,line,tips]{xy}

\renewcommand{\mod}{\operatorname{mod}}

\newcommand{\add}{\operatorname{add}}
\newcommand{\rad}{\operatorname{rad}}

\newcommand{\Hom}{\operatorname{Hom}}
\newcommand{\End}{\operatorname{End}}

\newcommand{\Ker}{\operatorname{Ker}}

\newcommand{\soc}{\operatorname{soc}}

\renewcommand{\Im}{\operatorname{Im}}

\newcommand{\repdim}{\operatorname{rep.dim.}}
\newcommand{\gldim}{\operatorname{gl.dim.}}

\newcommand{\cC}{\mathcal{C}}

\newcommand{\bA}{\mathbb{A}}

\newcommand{\bZ}{\mathbb{Z}}

\makeatletter
\def\dddots{\mathinner{\mkern1mu\raise1\p@
    \vbox{\kern13\p@\hbox{.}}\mkern2mu
    \raise4\p@\hbox{.}\mkern2mu\raise7\p@\hbox{.}\mkern1mu}}
\makeatother

\newtheorem{theorem}{Theorem}[section]

\newtheorem{lemma}[theorem]{Lemma}
\newtheorem{proposition}[theorem]{Proposition}
\newtheorem{corollary}[theorem]{Corollary}

\newtheorem{maintheorem}{Theorem}

\newtheorem{maincorollary}[maintheorem]{Corollary}

\theoremstyle{definition}

\newtheorem*{definition*}{Definition}

\newtheorem{example}[theorem]{Example}

\journal{Journal of Pure and Applied Algebra}

\begin{document}

\begin{frontmatter}



\title{Invariance of representation dimension under socle equivalence of selfinjective algebras}


\author[auth1]{Ibrahim Assem}
    \ead{ibrahim.assem@usherbrooke.ca}

    \address[auth1]{D\'epartement de math\'ematiques, Facult\'e des sciences, Universit\'e de Sherbrooke,
            Sherbrooke, Qu\'ebec J1K 2R1, Canada.}

\author[auth2]{Andrzej Skowro\'nski}
    \ead{skowron@mat.uni.torun.pl}
    \address[auth2]{Faculty of Mathematics and Computer Science, Nicolaus Copernicus University,
           Chopina 12/18, 87-100 Toru\'n, Poland.}
\author[auth3]{Sonia Trepode\corref{cor1}}
    \ead{strepode@mdp.edu.ar}
    \address[auth3]{Centro Marplatense de Investigaciones Matem\'aticas (CEMIM), Facultad de Ciencias Exactas y Naturales, Funes 3350,
            Universidad Nacional de Mar del Plata. Conicet, 7600 Mar del Plata, Argentina.}
\cortext[cor1]{Corresponding author}

%
%

\begin{abstract}
We prove that the representation dimension of finite dimensional
selfinjective algebras over a field is invariant under socle equivalence
and derive some consequences.
\end{abstract}


\end{frontmatter}

\begin{center}
\vspace*{-60.75mm}
\textit{Dedicated to Idun Reiten on the occasion of her 75th birthday}
\vspace*{52.75mm}
\end{center}


\section{Introduction}
\label{sec:intro}

Homological invariants are used to measure how far does
an algebra, or a module, deviates from a situation considered
to be ideal.
From the point of view of representation theory, one of the
most interesting and mysterious homological invariants
is the representation dimension of an Artin algebra,
introduced by Maurice Auslander in the early seventies \cite{Au}
and meant to measure the complexity of the morphisms in a module
category.
Interest in this invariant was revived when Igusa and Todorov
proved that algebras of representation dimension three have
 finite finitistic dimension \cite{IT}.
Iyama proved that the representation
dimension of an Artin algebra is always finite \cite{I}
and Rouquier that there exist algebras of arbitrarily
large representation dimension \cite{R}.
One important question is to identify which
algebraic procedures leave the representation dimension invariant.
It is known that stable equivalence preserves the representation
dimension, a result proved independently by Dugas \cite{D}
and Guo \cite{Gu}.
While derived equivalence does not, in general, preserve
the representation dimension, this is the case for selfinjective
algebras \cite{X}.

Our objective in this paper is to prove that the representation
dimension is preserved under socle equivalence of selfinjective
algebras.
We recall that two finite dimensional selfinjective algebras
$A$ and $A'$ over an arbitrary field $K$ are called socle equivalent
provided the quotient algebras $A / \soc A$ and $A' / \soc A'$
are isomorphic.
Socle equivalence plays a prominent r\^ole in the representation
theory of selfinjective algebras.
Frequently, interesting selfinjective algebras are socle equivalent
to others for whom
the representation theory and related invariants are well-understood.
For some results in this direction we refer the reader to
\cite{BiS},
\cite{BoS1},
\cite{ES2},
\cite{EHIS},
\cite{S3},
\cite{SY1}
or
\cite{SY2}.
Our main theorem is then the following.

\begin{maintheorem}
\label{th:A}
Let $A, A'$ be basic and connected socle equivalent
selfinjective algebras.
Then $A$ and $A'$ have the same representation dimension.
\end{maintheorem}

Our proof is constructive:
given an Auslander generator for $\mod A$,
we show how to construct one for $\mod B$.

Auslander's expectation was that the representation dimension
would provide a reasonable way to measure how far an algebra is from
being representation finite.
Because an algebra is representation finite if and only if
it has representation dimension two \cite{Au}, algebras
of representation dimension three present a special interest.
In this line of ideas, we apply Theorem~\ref{th:A}
to selfinjective algebras of tilted type.
We recall that a selfinjective algebra $A$ is called of tilted type
if there exists a tilted algebra $B$ such that $A$ is an orbit
category of the repetitive category $\widehat{B}$ of $B$,
in the sense of \cite{HW}.
As a first consequence of Theorem~\ref{th:A} and the results
of \cite{AST1}, \cite{AST2},
we show that, if $A$ is a selfinjective algebra socle equivalent
to a representation-infinite selfinjective algebra of tilted type,
then the representation dimension of $A$ equals three.

We next turn to the problem of relating the representation
dimension with the shape of Auslander-Reiten components.
Using the notation of stable slice introduced in \cite{SY6},
we prove our second main result.

\begin{maintheorem}
\label{th:B}
Let $A$ be a representation-infinite basic and connected
selfinjective algebra admitting a $\tau_A$-rigid  stable slice
in its Auslander-Reiten quiver.
Then the representation dimension of $A$ equals three.
\end{maintheorem}

This theorem entails the following interesting corollary.

\begin{maincorollary}
\label{cor:C}
Let $A$ be a basic and connected
selfinjective algebra admitting an acyclic generalised
standard Auslander-Reiten component.
Then the representation dimension of $A$ equals three.
\end{maincorollary}

We now describe the contents of the paper.
After a preliminary Section~\ref{sec:preliminary}
in which we fix the notation and recall facts
on the representation dimension and socle equivalence,
we prove our Theorem~\ref{th:A} in Section~\ref{sec:invariance}.
Section~\ref{sec:selfinjective} is devoted to the application to
the selfinjective algebras of tilted type, including
Theorem~\ref{th:B} and Corollary~\ref{cor:C}.
Finally, Section~\ref{sec:examples} consists of illustrative examples.

\section{Representation dimension and socle equivalence}
\label{sec:preliminary}

\subsection{Notation.}
Throughout this paper,
$K$ denotes an arbitrary (commutative) field.
By an algebra $A$ is meant a basic, connected,
associative finite dimensional $K$-algebra.
Modules are finitely generated right $A$-modules,
and we denote by $\mod A$ their category.
For a module $M$, we denote by $\ell(M)$ its
composition length.
The notation $\add M$ stands for the additive
full subcategory of $\mod A$ having as objects
the direct sums of direct summands of $M$.
Given a full subcategory $\cC$ of $\mod A$,
we sometimes write $M \in \cC$ to express that
$M$ is an object in $\cC$.

We use freely standard results of representation theory,
for which we refer to \cite{ASS,ARS,SY5,SY7}.

\subsection{Representation dimension.}
The notion of representation dimension was introduced
in \cite{Au}.
It is defined as follows.
\begin{definition*}
Let $A$ be a non-semisimple algebra.
Its \emph{representation dimension} $\repdim A$
is the infinimum of the global dimensions of the algebras
$\End M$, where $M$ ranges over all $A$-modules
which are at the same time generators and cogenerators
of $\mod A$.
\end{definition*}

If $M$ is a generator-cogenerator of $\mod A$ for which
$\repdim A = $ \newline $\gldim \End M$, then
$M$ is called an \emph{Auslander generator} for $\mod A$.

For instance, if $A$ is a selfinjective algebra, then the module
$A_A$ is at the same time a generator and a cogenerator
of $\mod A$.
In fact, in this case, an arbitrary generator-cogenerator $M$
of $\mod A$ can be assumed to be of the form $M = A \oplus N$,
where the module $N_A$ has no projective direct summand.

In order to give a criterion allowing to compute the
representation dimension, we need to recall a few definitions
and facts.

Let $M$ be a fixed $A$-module.
Given a module $X$, a morphism
$f_0 : M_0 \to X$
with $M_0 \in \add M$ is called a \emph{right $\add M$-approximation}
of $X$ provided the induced morphism
\[
  \Hom_A(M,f_0) : \Hom_A(M,M_0) \to \Hom_A(M,X)
\]
is surjective.
Note that, if $M$ generates $X$, then any right $\add M$-approximation
of $X$ is surjective.

A right $\add M$-approximation $f_0 : M_0 \to X$
is called \emph{right minimal}
if any morphism  $g : M_0 \to X$
such that $f_0 g = f_0$ is an isomorphism
\[
  \xymatrix@C=6pc@R=1.25pc{
   M_0 \ar[r]^{f_0} \ar@{->}[d]_g & X \ar@{=}[d] \\
   M_0 \ar[r]^{f_0}  & X \, . \!\!\! \\
  }
\]

It is a \emph{right minimal $\add M$-approximation}
of $X$ if it is a right $\add M$-approximation
of $X$ and it is right minimal.

It turns out that any right $\add M$-approximation
admits a direct summand which is a right minimal $\add M$-approximation.

\begin{lemma}
\label{lem:2.1}
Let $M,X$ be modules and $f_1 : M_1 \to X$,
with $M_1 \in \add M$, be a right $\add M$-approximation of $X$.
Then
\begin{enumerate}[(a)]
 \item
  There exists a direct summand $M_0$ of $M_1$ such that
  $f_0 = f_1|_{M_0} : M_0 \to X$
  is a right minimal $\add M$-approximation.
 \item
  For any right minimal $\add M$-approximation
  $f_0 : M_0 \to X$, there exists a section  $s : M_0 \to M_1$
  such that $f_1 s = f_0$.
\end{enumerate}
\end{lemma}

\begin{proof}
(a) This is just \cite[Theorem~I.2.2]{ARS}.

\smallskip

(b)
Because $M_0, M_1$ are right $\add M$-approximations,
there exist morphisms
$s : M_0 \to M_1$
and
$r : M_1 \to M_0$
making the following diagram commutative
\[
  \xymatrix@C=6pc@R=1.25pc{
   M_0 \ar[r]^{f_0} \ar@{->}[d]_s & X \ar@{=}[d] \\
   M_1 \ar[r]^{f_1} \ar@{->}[d]_r & X \ar@{=}[d] \\
   M_0 \ar[r]^{f_0}  & X \, . \!\!\! \\
  }
\]
Because $f_0$ is minimal, $r s$  is an isomorphism.
Therefore, $s$ is a section.
\end{proof}

An exact sequence of the form
\[
  0 \to
  M_d \xrightarrow{f_d}
  M_{d-1} \to
  \dots \to
  M_1 \xrightarrow{f_1}
  M_0 \xrightarrow{f_0}
  X \to
  0
\]
with all $M_i$ lying in $\add M$ is called a
\emph{right $\add M$-approximation resolution}
of $X$ with length $d$ provided the induced sequence
\[
  0 \to
  \Hom_A(M,M_d) \to
  \Hom_A(M,M_{d-1}) \to
  \dots \to
  \Hom_A(M,M_0) \to
  \Hom_A(M,X) \to
  0
\]
is exact.
It is called a
\emph{right minimal $\add M$-approximation resolution}
of $X$ if moreover each of the morphisms
$f_i : M_i \to \Im f_i$,
with $i \geqslant 0$,
is right minimal.

The following statement is a consequence of
Lemma~\ref{lem:2.1}.

\begin{corollary}
\label{cor:2.2}
Let $M,X$ be given modules and
\[
  0 \to
  N_d \xrightarrow{g_d}
  N_{d-1} \to
  \dots \to
  N_1 \xrightarrow{g_1}
  N_0 \xrightarrow{g_0}
  X \to
  0
\]
a right $\add M$-approximation resolution of $X$.
Then there exists a right minimal $\add M$-approximation
resolution of $X$
\[
  0 \to
  M_d \xrightarrow{f_d}
  M_{d-1} \to
  \dots \to
  M_1 \xrightarrow{f_1}
  M_0 \xrightarrow{f_0}
  X \to
  0
\]
which is a direct summand of the first one.
\end{corollary}

\begin{proof}
We construct the second sequence by induction.
Because of Lemma~\ref{lem:2.1},
there exists a right minimal $\add M$-approximation
$f_0 : M_0 \to X$ and two morphisms
$s_0 : M_0 \to N_0$,
$r_0 : N_0 \to M_0$
such that the right squares of the following diagram commute
\[
  \xymatrix@C=2pc@R=1.25pc{
   0 \ar[r] & Y_0 \ar[rr]^{j_0} \ar[d]^{u_0} &&
                     M_0 \ar[rr]^{f_0} \ar[d]^{s_0} &&
                     X \ar[r] \ar@{=}[d] & 0 \\
   0 \ar[r] & Z_0 \ar[rr]^{i_0} \ar[d]^{v_0} &&
                     N_0 \ar[rr]^{g_0} \ar[d]^{r_0} &&
                     X \ar[r] \ar@{=}[d] & 0 \\
   0 \ar[r] & Y_0 \ar[rr]^{j_0} &&
                     M_0 \ar[rr]^{f_0} &&
                     X \ar[r] & 0 \, . \!\!\! \\
  }
\]
Letting $Y_0, Z_0$ be respectively the kernels of $f_0, g_0$
and $i_0,j_0$
the inclusion morphisms, one gets $u_0,v_0$ by
passing to the kernels.
Because $f_0$ is minimal, $r_0 s_0$  is an isomorphism,
hence so is $v_0 u_0$.
In particular, $s_0$ and $u_0$ are sections.
Also $Z_0 = \Im g_1$,
and so there exists an epimorphism
$q_1 : N_1 \to Z_0$ such that  $g_1 = i_0 q_1$.

Again,  Lemma~\ref{lem:2.1}
yields a right minimal $\add M$-approximation
$p_1 : M_1 \to Y_0$
and morphisms
$s_1 : M_1 \to N_1$,
$r_1 : N_1 \to M_1$
and
$t_1 : M_1 \to M_1$
such that
the right squares of the following diagram commute
\[
  \xymatrix@C=2pc@R=1.25pc{
   0 \ar[r] & Y_1 \ar[rr]^{j_1} \ar[d]^{u_1} &&
                     M_1 \ar[rr]^{p_1} \ar[d]^{s_1} &&
                     Y_0 \ar[r] \ar[d]^{u_0} & 0 \\
   0 \ar[r] & Z_1 \ar[rr]^{i_1} \ar[d]^{v_1} &&
                     N_1 \ar[rr]^{q_1} \ar[d]^{r_1} &&
                     Y_0 \ar[r] \ar[d]^{v_0} & 0 \\
   0 \ar[r] & Y_1 \ar[rr]^{j_1} \ar[d]^{w_1} &&
                     M_1 \ar[rr]^{p_1} \ar[d]^{t_1} &&
                     Y_0 \ar[r] \ar[d]^{(v_0 u_0)^{-1}} & 0 \\
   0 \ar[r] & Y_1 \ar[rr]^{j_1} &&
                     M_1 \ar[rr]^{p_1} &&
                     Y_0 \ar[r] & 0  \, . \!\!\! \\
  }
\]
Letting $Y_1, Z_1$ be respectively the kernels of $p_1, q_1$
and $i_1,j_1$
the inclusion morphisms, one gets $u_1,v_1,w_1$ by
passing to the kernels.
Because $p_1$ is right minimal, $t_1 r_1 s_1$ and $w_1 v_1 u_1$
are isomorphisms.
Therefore $s_1$ and $u_1$ are sections
while $t_1$, $w_1$ are retractions,
and hence isomorphisms.
Moreover, we have
\[
  g_1 s_1
   = i_0 q_1 s_1
   = i_0 u_0 p_1
   = s_0 j_0 p_1 .
\]
Setting
$f_1 = j_0 p_1 : M_1 \to M_0$,
we obtain the second morphism
in the required right minimal $\add M$-approximation
resolution.
Continuing in this way, we construct the wanted right
minimal $\add M$-approximation
resolution and sections  $s_i : M_i \to N_i$
such that the following diagram commutes
\[
  \xymatrix@C=2pc@R=1.25pc{
   0 \ar[r] & M_d \ar[r]^{f_d} \ar[d]^{s_d} &
                     M_{d-1} \ar[r]^{f_{d-1}} \ar[d]^{s_{d-1}} &
                    \cdots \ar[r] &
                     M_{1} \ar[r]^{f_{1}} \ar[d]^{s_{1}} &
                     M_{0} \ar[r]^{f_{0}} \ar[d]^{s_{0}} &
                     X \ar[r] \ar@{=}[d] &
                    0 \\
   0 \ar[r] & N_d \ar[r]^{g_d} &
                     N_{d-1} \ar[r]^{g_{d-1}} &
                    \cdots \ar[r] &
                     N_{1} \ar[r]^{g_{1}} &
                     N_{0} \ar[r]^{g_{0}} &
                     X \ar[r] &
                    0 \, . \!\!\! \\
  }
\]
\end{proof}

We need essentially the following characterisation
of the representation dimension for which we refer to
\cite{CP},
\cite{EHIS}.

\begin{theorem}
\label{th:2.3}
Let $A$ be an algebra.
Then $\repdim A \leqslant d+2$
if and only if there exists a generator-cogenerator $M$
such that every $A$-module $X$ has a right $\add M$-approximation
resolution
of length $d$.
\end{theorem}

Because of Corollary~\ref{cor:2.2},
we may reformulate this theorem by saying that
$\repdim A \leqslant d+2$ if and only if
there exists a generator-cogenerator $M$ such that
every $A$-module $X$ has a right minimal
$\add M$-approximation resolution of length $d$.

\subsection{Socle equivalence.}
Let $A$ and $A'$ be basic and connected
selfinjective algebras.
Then $A$ and $A'$ are called \emph{socle equivalent}
if the quotient algebras $A/\soc A$ and $A'/\soc A'$
are isomorphic.
Recall indeed that the socle of a selfinjective algebra
is a two-sided ideal, see \cite[Corollary~IV.6.14]{SY5}.

For simplicity of notation, if $A$, $A'$ are socle equivalent,
then we treat the isomorphism $A/\soc A \cong A'/\soc A'$
as an identification, that is, we always assume that
$A/\soc A = A'/\soc A'$.

If $A$ and $A'$ are socle equivalent, then they have
the same ordinary quiver.
They also have very similar Auslander-Reiten quivers.
Indeed, the indecomposable nonprojective
$A$-modules coincide with the indecomposable
nonprojective $A'$-modules.
Furthermore, if $P$ is an indecomposable projective
$A$-module, then there exists an almost split sequence
\[
  0 \to \rad P \to
  P \oplus (\rad P / \soc P)
  \to P / \soc P \to 0
\]
in $\mod A$.
The $A$-modules
$\rad P$,
$\rad P / \soc P$
and
$P / \soc P$
are clearly annihilated by the socle of $A$,
so they are $A/\soc A$-modules.
For the same reason, $P$ is not an $A/\soc A$-module.
Moreover, the irreducible morphisms
$\rad P \to \rad P / \soc P \to P / \soc P$
in $\mod A$,
remain irreducible in $\mod (A/\soc A)$,
see \cite[p.~186]{ARS}.
As an $A/\soc A$-module, $P/\soc P$
is indecomposable projective, while $\rad P$
is the injective envelope of $\soc P$.
Finally, because $A,A'$ are socle equivalent
and so have the same ordinary quiver,
there exists a unique indecomposable projective
$A'$-module $P'$ such that
$P/\soc P = P'/\soc P'$.
The relation between $P$ and $P'$ is
described in the following lemma.

\begin{lemma}
\label{lem:2.4}
Let $A$, $A'$ be socle equivalent selfinjective algebras,
and $P_A$, $P'_{A'}$, be indecomposable projective modules
such that $P/\soc P = P'/\soc P'$.
Then we have:
\begin{enumerate}[(a)]
 \item
  The almost split sequence in $\mod A'$
  having $P'$ as summand of the middle term is
  \[
    0 \to \rad P \to
    P' \oplus (\rad P / \soc P)
    \to P / \soc P \to 0 .
  \]
 \item
  $\rad P = \rad P'$.
 \item
  $\soc P = \soc P'$.
 \item
  $\ell(P) = \ell(P')$.
 \item
  If
  $j : \rad P \to P$,
  $j' : \rad P' \to P'$
  and
  $p : P \to P/\soc P$,
  $p' : P' \to P'/\soc P'$
  are the canonical morphisms,
  then $p j = p' j'$.
\end{enumerate}
\end{lemma}

\begin{proof}
(a)
This follows from the discussion before using the hypothesis that
$P/\soc P = P'/\soc P'$.

\smallskip

(b)
This follows immediately from (a).

\smallskip

(c)
$\soc P = \soc (\rad P) = \soc (\rad P') = \soc P'$.

\smallskip

(d)
$\ell(P) = \ell (\rad P) + \ell (P/\soc P) - \ell (\rad P/\soc P) = \ell(P')$.

\smallskip

(e)
The morphism $p j$ can be rewritten as the composition of the canonical
morphisms
$\rad P \to \rad P / \soc P \to P / \soc P$.
Similarly,
$p' j'$ is the composition of the corresponding morphisms
in $\mod A'$.
Because the modules through which they pass are the same,
this implies that $p j = p' j'$.
\end{proof}

\section{Invariance of the representation dimension}
\label{sec:invariance}

Our objective in this section is to prove that two socle equivalent
selfinjective algebras have the same representation
dimension.
Throughout this section, $A$ and $A'$ denote two
selfinjective algebras such that
$A/\soc A = A'/\soc A'$.

\begin{proposition}
\label{prop:3.1}
Let
$0 \to Y \xrightarrow{\left(\begin{smallmatrix}f\\g\end{smallmatrix}\right)}
 N_0 \oplus P \xrightarrow{\left(\begin{smallmatrix}u& v\end{smallmatrix}\right)}
 X \to 0$
be a short exact sequence in $\mod A$, with $P$ projective
and $X,Y,N_0$ having no projective direct summand.
Let $P'$ be the projective $A'$-module such that
$P' / \soc P' = P / \soc P$.
Then there exists a short exact sequence in $\mod A'$
\[
  0 \to Y \xrightarrow{\left(\begin{smallmatrix}f\\g'\end{smallmatrix}\right)}
   N_0 \oplus P' \xrightarrow{\left(\begin{smallmatrix}u& v'\end{smallmatrix}\right)}
   X \to 0
   .
\]
\end{proposition}

\begin{proof}
Because the morphism $g : Y \to P$ cannot be surjective,
its image lies in $\rad P$, and therefore we have a factorisation
\[
  \xymatrix@C=2pc@R=2pc{
   Y \ar[rr]^{\left(\begin{smallmatrix}f\\g\end{smallmatrix}\right)}
      \ar[rd]_(.4){\left(\begin{smallmatrix}f\\h\end{smallmatrix}\right)}
      && N_0 \oplus P \\
    & N_0 \oplus \rad P \ar[ru]_(.55){\left(\begin{smallmatrix}1&0\\0&j\end{smallmatrix}\right)}
  }
\]
where $j :\rad P \to P$ is the canonical inclusion and $g = j h$.

Because of Lemma~\ref{lem:2.4},
$\rad P = \rad P'$.
Let $j' : \rad P' \to P'$
be the canonical inclusion and set $g' = j' h$.
We get a composed morphism
\[
 \begin{pmatrix}f\\g'\end{pmatrix}
 =
 \begin{pmatrix}1&0\\0&j'\end{pmatrix}
 \begin{pmatrix}f\\h\end{pmatrix}
: Y \to
N_0 \oplus P'
\]
in $\mod A'$.
Because
$\big(\!\begin{smallmatrix}f\\g\end{smallmatrix}\!\big)$
is injective, so is
$\big(\!\begin{smallmatrix}f\\h\end{smallmatrix}\!\big)$,
hence so is
$\big(\!\begin{smallmatrix}f\\g'\end{smallmatrix}\!\big)$.

We now construct in a similar way a morphism
$N_0 \oplus P' \to X$.
Because the morphism $v : P \to X$ cannot be injective,
it factors through $P/\soc P$.
Therefore, we have a factorisation
\[
  \xymatrix@C=2pc@R=2pc{
   N_0 \oplus P \ar[rr]^{\left(\begin{smallmatrix}u&v\end{smallmatrix}\right)}
      \ar[rd]_(.4){\left(\begin{smallmatrix}1&0\\0&p\end{smallmatrix}\right)}
      && X \\
    & N_0 \oplus P/\soc P \ar[ru]_(.55){\left(\begin{smallmatrix}u&w\end{smallmatrix}\right)}
  }
\]
where $p : P \to P / \soc P$ is the canonical projection
and $v = w p$.

Now, $P / \soc P = P' / \soc P'$.
Thus we get a composed morphism
\[
 \begin{pmatrix}u&v'\end{pmatrix}
 =
 \begin{pmatrix}u&w\end{pmatrix}
 \begin{pmatrix}1&0\\0&p'\end{pmatrix}
: N_0 \oplus P' \to X
\]
in $\mod A'$,
where
$p' : P' \to  P' / \soc P'$
is the canonical projection and $v' = w p'$.
Because $(u\ \, v)$ is surjective, so is
$(u\ \, w)$, hence so is $(u\ \, v')$.

The construction is encoded in the following
commutative diagram
\[
  \xymatrix@C=1.5pc@R=1.25pc{
  0 \ar[r] &
   Y \ar[rr]^{\left(\begin{smallmatrix}f\\g\end{smallmatrix}\right)}
      \ar[rd]_(.4){\left(\begin{smallmatrix}f\\h\end{smallmatrix}\right)}
   \ar@{=}[ddd]
      && N_0 \oplus P  \ar[rr]^{\left(\begin{smallmatrix}u&v\end{smallmatrix}\right)}
      \ar[rd]_(.4){\left(\begin{smallmatrix}1&0\\0&p\end{smallmatrix}\right)}
      && X
   \ar@{=}[ddd]
  \ar[r] & 0
\\ &
    & N_0 \oplus \rad P \ar[ru]_(.55)
{\left(\begin{smallmatrix}1&0\\0&j\end{smallmatrix}\right)}
   \ar@{=}[d]
  && N_0 \oplus P/\soc P \ar[ru]_(.55)
{\left(\begin{smallmatrix}u&w\end{smallmatrix}\right)}
   \ar@{=}[d]
\\ &
    & N_0 \oplus \rad P' \ar[rd]^(.55)
{\left(\begin{smallmatrix}1&0\\0&j'\end{smallmatrix}\right)}
  && N_0 \oplus P'/\soc P' \ar[rd]^(.55)
{\left(\begin{smallmatrix}u&w\end{smallmatrix}\right)}
\\
  0 \ar[r] &
   Y \ar[rr]_{\left(\begin{smallmatrix}f\\g'\end{smallmatrix}\right)}
      \ar[ru]^(.4){\left(\begin{smallmatrix}f\\h\end{smallmatrix}\right)}
      && N_0 \oplus P'  \ar[rr]_{\left(\begin{smallmatrix}u&v'\end{smallmatrix}\right)}
      \ar[ru]^(.4){\left(\begin{smallmatrix}1&0\\0&p'\end{smallmatrix}\right)}
      && X
  \ar[r] & 0
\,.\!\!\!
  }
\]
We know that the upper sequence is exact, and we want
to prove that so is the lower sequence.

We have already proven that
$\big(\!\begin{smallmatrix}f\\g'\end{smallmatrix}\!\big)$
is injective, and that
$(u\ \, v')$ is surjective,
so we only need to check that
$\Ker (u\ \, v') = \Im \big(\!\begin{smallmatrix}f\\g'\end{smallmatrix}\!\big)$.
First, we have
\begin{align*}
  (u\ \, v') \begin{pmatrix}f\\g'\end{pmatrix}
  &= u f + v' g'
  = u f + w p' j' h
  = u f + w p j h
\\ &
  = u f + v g
  = (u\ \, v) \begin{pmatrix}f\\g\end{pmatrix}
  = 0 ,
\end{align*}
where we have used Lemma~\ref{lem:2.4}~(e).

In order to prove that
$\Ker (u\ \, v') \subseteq \Im \big(\!\begin{smallmatrix}f\\g'\end{smallmatrix}\!\big)$,
let $x \in N_0$ and $a \in P'$ be such that
$\big(\!\begin{smallmatrix}x\\a\end{smallmatrix}\!\big) \in \Ker (u\ \, v')$.
Because $p$ and $p'$ are surjective,
there exists $b \in P$ such that
$p(b) = p'(a)$, so
$\big(\!\begin{smallmatrix}x\\b\end{smallmatrix}\!\big) \in N_0 \oplus P$.
We have
\begin{align*}
  (u\ \, v) \begin{pmatrix}x\\b\end{pmatrix}
  &= u (x) + v(b)
  = u (x) + w p (b)
  = u (x) + w p' (a)
  = u (x) + v' (a)
\\ &
  = (u\ \, v') \begin{pmatrix}x\\a\end{pmatrix}
  = 0 ,
\end{align*}
so that
$\big(\!\begin{smallmatrix}x\\b\end{smallmatrix}\!\big) \in \Ker (u\ \, v)$.
Therefore, there exists $y_0 \in Y$ such that
\[
 \begin{pmatrix}x\\b\end{pmatrix}
  =
 \begin{pmatrix}f\\g\end{pmatrix}
 (y_0),
\]
that is, such that
$x = f(y_0)$
and
$b = g(y_0)$.

Now, we have
\begin{align*}
  p'\big(a - g'(y_0)\big)
  &= p'(a) - p' g'(y_0)
    = p'(a) - p'j'h(y_0)
    = p'(a) - pjh(y_0)
\\ &
    = p'(a) - p g(y_0)
    = p'(a) - p(b)
   = 0 .
\end{align*}
Therefore $a - g'(y_0) \in \Ker p' = \soc P'$.
Because of
Lemma~\ref{lem:2.4}~ (c),
$a - g'(y_0) \in \soc P$
and so
$p(a - g'(y_0)) = 0$ .
Therefore
\[
  (u\ \, v) \begin{pmatrix}0\\a - g'(y_0)\end{pmatrix}
 =
    (u\ \, w)
     \begin{pmatrix}1&0\\0&p\end{pmatrix}
     \begin{pmatrix}0\\a - g'(y_0)\end{pmatrix}
  = 0 .
\]
Exactness of the upper row yields a $c \in Y$ such that
\[
  \begin{pmatrix}0\\a - g'(y_0)\end{pmatrix}
 =
     \begin{pmatrix}f\\g\end{pmatrix} (c) .
\]
This means that $f(c) = 0$ and $g(c) = a - g'(y_0)$.
Because $a - g'(y_0) \in \soc P \subseteq \rad P$,
we have $g(c) = j h(c) \in \rad P$,
and therefore $j h (c) = h (c)$.
Hence $a = g'(y_0) + h(c)$.

Set $y = y_0 + c \in Y$.
Because $f(c) = 0$, we have
$f(y) = f(y_0) + f(c) = f(y_0) = x$.
On the other hand,
\[
  g'(y) = g'(y_0) + g'(c) = g'(y_0) + j' h(c) .
\]
Because $h(c) \in \rad P = \rad P'$,
we have $j' h(c) = h(c)$ and so
\[
  g'(y) = g'(y_0) + h(c) = a .
\]
We have proved that
$\big(\!\begin{smallmatrix}x\\a\end{smallmatrix}\!\big)
 = \big(\!\begin{smallmatrix}f\\g'\end{smallmatrix}\!\big) (y)
 \in \Im \big(\!\begin{smallmatrix}f\\g'\end{smallmatrix}\!\big)$,
as required.
\end{proof}

Let now $M$ be an Auslander generator for $\mod A$.
Because $A$ is selfinjective, we can assume that $M$
is of the form
\[
  M = N \oplus A ,
\]
where $N$ has no projective summands.
But then $N$ is also an $A'$-module.
We claim that
\[
  M' = N \oplus A'
\]
is an Auslander generator for $\mod A'$.
The first step in the proof is the following lemma.

\begin{lemma}
\label{lem:3.2}
Let
$0 \to Y \xrightarrow{\left(\begin{smallmatrix}f\\g\end{smallmatrix}\right)}
 N_0 \oplus P \xrightarrow{\left(\begin{smallmatrix}u& v\end{smallmatrix}\right)}
 X \to 0$
be an exact sequence in $\mod A$,
with $P$ projective,
$N_0 \in \add N$
and $X,Y$ having no projective direct summands.
Assume that $(u\ \, v)$ is a right $\add M$-approximation
in $\mod A$.
Then, in the corresponding exact sequence
\[
  0 \to Y \xrightarrow{\left(\begin{smallmatrix}f\\g'\end{smallmatrix}\right)}
   N_0 \oplus P' \xrightarrow{\left(\begin{smallmatrix}u& v'\end{smallmatrix}\right)}
   X \to 0
\]
in $\mod A'$,
the morphism $(u\ \, v')$ is a right $\add M'$-approximation
in $\mod A'$.
\end{lemma}

\begin{proof}
Let $M'_0$ be an indecomposable direct summand of $M'$.
We claim that any morphism $f'_0 : M'_0 \to X$
lifts to $N_0 \oplus P'$.
Because  $M'_0$ is indecomposable, either
$M'_0$ is projective, in which case the statement is obvious,
or else $M'_0 \in \add N$.
In this latter case, $f'_0$ is also a morphism in $\mod A$.
Therefore there exists
$\big(\!\begin{smallmatrix}t_1\\t_2\end{smallmatrix}\!\big)
 : M'_0 \to N_0 \oplus P$
such that
$(u\ \, v) \big(\!\begin{smallmatrix}t_1\\t_2\end{smallmatrix}\!\big) = f'_0$.
Because $M'_0 \in \add N$,
the morphism $t_2 : M'_2 \to P$ cannot be surjective,
hence it factors through  $\rad P$.
That is, there exists $t_2^* : M'_2 \to \rad P$ such that
$t_2 = j t_2^*$ where $j$ is, as before,
the canonical inclusion.
But $\rad P = \rad P'$.
Letting $j' : \rad P' \to P'$ be the canonical inclusion,
we set $t'_2 = j' t_2^*$.
We claim that
$(u\ \, v') \big(\!\begin{smallmatrix}t_1\\t'_2\end{smallmatrix}\!\big) = f'_0$.
Indeed, denoting, as before,
by $p : P \to P / \soc P$ and
$p' : P' \to P' / \soc P'$
the canonical projections, we have
\begin{align*}
  (u\ \, v') \begin{pmatrix}t_1\\t'_2\end{pmatrix}
  &= u t_1 + v' t'_2
  = u t_1 + w p' j' t_2^*
  = u t_1 + w p j t_2^*
\\ &
  = u t_1 + v t_2
  = (u\ \, v) \begin{pmatrix}t_1\\t_2\end{pmatrix}
  = f'_0 ,
\end{align*}
where we have used Lemma~\ref{lem:2.4}~(e).
\end{proof}

\begin{lemma}
\label{lem:3.3}
In the notation of
Lemma~\ref{lem:3.2},
if $(u\ \, v)$ is minimal,
then so is $(u\ \, v')$.
\end{lemma}

\begin{proof}
Assume that $(u\ \, v')$ is not minimal.
Then there exists a minimal approximation
$(u_1\ \, v'_1) : N_1 \oplus P'_1 \to X$
with $N_1 \in \add N$,
$P'_1$ projective, and a commutative diagram
\[
  \xymatrix@C=4pc@R=.5pc{
   N_1 \oplus P'_1
      \ar[rd]^(.55){\left(\begin{smallmatrix}u&v'_1\end{smallmatrix}\right)}
      \ar[dd]_{s} \\
    & X \\
    N_0 \oplus P'
      \ar[ru]_(.55){\left(\begin{smallmatrix}u&v'\end{smallmatrix}\right)}
  }
\]
where $s$ is a proper section,
see Lemma~\ref{lem:2.1}.
Let $P_1$ be the projective $A$-module such that
$P_1 / \soc P_1 = P'_1 / \soc P'_1$.
Exchanging the r\^oles of $A$ and $A'$ in
Lemma~\ref{lem:3.2},
we get a right $\add M$-approximation
\[
  (u_1\ \, v_1) : N_1 \oplus P_1 \to X
\]
in $\mod A$.
Now, we have $\ell(P_1) = \ell(P'_1)$
because of
Lemma~\ref{lem:2.4}~(d),
and therefore
\begin{align*}
  \ell(N_1 \oplus P_1)
  &= \ell(N_1) + \ell(P_1)
  = \ell(N_1) + \ell(P'_1)
\\ &
  = \ell(N_1 \oplus P'_1)
  < \ell(N_0 \oplus P')
  = \ell(N_0 \oplus P)
  ,
\end{align*}
and this contradicts the minimality of
$(u\ \, v) : N_0 \oplus P \to X$.
Therefore, $s$ is not proper and so
$(u_1\ \, v'_1) : N_1 \oplus P'_1 \to X$
is a right minimal $\add M'$-approximation.
\end{proof}

We are now able to prove our Theorem~\ref{th:A}.

\begin{theorem}
\label{th:3.4}
Let $A, A'$ be socle equivalent  basic and connected
selfinjective algebras.
Then $\repdim A = \repdim  A'$.
Furthermore, if $M = N \oplus A$ is an Auslander generator
for $\mod A$, with $N$ having no projective direct summands,
then $M' = N \oplus A'$ is an Auslander generator for $\mod A'$.
\end{theorem}

\begin{proof}
For simplicity, we may assume $A/\soc A = A'/\soc A'$.
Let $X$ be an indecomposable nonprojective $A$-module,
and let
\[
  0 \to
  N_d \to
  N_{d-1} \oplus P_{d-1} \to
  \dots \to
  N_0 \oplus P_{0} \to
  X \to
  0
\]
be a right minimal $\add M$-approximation resolution,
with $N_i \in \add N$ and the $P_i$ projective for all $i$.
Notice that the minimality of this sequence implies that
the last nonzero term on the left has no projective direct
summand, and therefore belongs to $\add N$.

For each $i$, let $P'_i$ be the projective $A'$-module such that
$P'_i / \soc P'_i = P_i / \soc P_i$.
We claim that the corresponding sequence
\[
  0 \to
  N_d \to
  N_{d-1} \oplus P'_{d-1} \to
  \dots \to
  N_0 \oplus P'_{0} \to
  X \to
  0
\]
is a right minimal $\add M'$-approximation resolution in $\mod A'$.

We prove this claim by induction.
Let first
$(u\ \, v) : N_0 \oplus P_{0} \to  X$
be a right minimal $\add M$-approximation.
Because $M$ is a generator of $\mod \Lambda$,
the morphism $(u\ \, v)$ is injective.
Letting $Y = \Ker(u\ \, v)$, we have a short exact sequence
\[
  0 \to
  Y \to
  N_0 \oplus P_{0} \xrightarrow{(u\ \, v)}
  X \to
  0
\]
in $\mod A$.
If $Y$ has a projective (= injective) direct summand,
then this summand splits off and we have a contradiction
to the minimality of $(u\ \, v)$.
Therefore $Y$ has no projective direct summand.
Applying
Proposition~\ref{prop:3.1},
we get a short exact sequence
\begin{gather}
  \tag{*}
  \label{eq:*}
  0 \to
  Y \to
  N_0 \oplus P'_{0} \xrightarrow{(u\ \, v')}
  X \to
  0
\end{gather}
in $\mod A'$.
Because of Lemmata \ref{lem:3.2} and \ref{lem:3.3},
$(u\ \, v') : N_0 \oplus P'_{0} \to  X$
is a right minimal $\add M$-approximation
in $\mod A'$.

Now, we have a right minimal $\add M$-approximation
resolution of $Y$ in $\mod A$
\[
  0 \to
  N_d \to
  N_{d-1} \oplus P_{d-1} \to
  \dots \to
  N_1 \oplus P_{1} \to
  Y \to
  0
 .
\]
The induction hypothesis yields a right minimal
$\add M'$-approximation resolution of $Y$ in $\mod A'$
\begin{gather}
  \tag{**}
  \label{eq:**}
  0 \to
  N_d \to
  N_{d-1} \oplus P'_{d-1} \to
  \dots \to
  N_1 \oplus P'_{1} \to
  Y \to
  0
  .
\end{gather}
Splicing the sequences (\ref{eq:*}) and (\ref{eq:**})
yields the desired right minimal $\add M'$-approximation
resolution of $X$ in $\mod A'$.
This establishes our claim.

The statement of the theorem now follows easily from
the claim and Theorem~\ref{th:2.3}.
\end{proof}

The reader will observe that, in the course of the proof, we
have constructed a bijection between right minimal
$\add M$-approximation resolutions in $\mod A$ and
right minimal $\add M'$-approximation resolutions in $\mod A'$.

\section{Selfinjective algebras of tilted type}
\label{sec:selfinjective}

In this section, we present some applications of the main
result of the paper to selfinjective algebras, which are socle
equivalent to selfinjective algebras of tilted type.
For background on hereditary and tilted
algebras over arbitrary fields we
refer to \cite[Chapters VII and VIII]{SY7}.
We also refer to \cite{SY4} for general results on
selfinjective algebras of tilted type.

Let $B$ be a basic finite dimensional $K$-algebra and
$1 = e_1 + \dots + e_n$ be a decomposition of the identity
of $B$ into a complete sum of primitive
orthogonal idempotents.
We associate to $B$ a selfinjective locally bounded $K$-category
$\widehat{B}$, called its \emph{repetitive category} \cite{HW}.
The objects of $\widehat{B}$ are the
$e_{m,i}$, with $m \in \bZ$ and $i \in \{1,\dots,n\}$,
and the morphism spaces are defined by
\[
  \widehat{B} \big(e_{k,i}, e_{s,j}\big) =
  \left\{
   \begin{array}{cl}
     e_j B e_i &\mbox{if } k = s \\
     D(e_i B e_j) &\mbox{if } k = s - 1 \\
     0 &\mbox{otherwise.}  \\
   \end{array}
  \right.
\]
We denote by $\nu_{\widehat{B}}$ the so-called
\emph{Nakayama automorphism} of $\widehat{B}$
defined by
\[
  \nu_{\widehat{B}} (e_{m,i}) = e_{m+1,i}
\]
for all $(m,i)$.
A group $G$ of $K$-linear automorphisms of the category
$\widehat{B}$ is said to be \emph{admissible}
if $G$ acts freely on the objects of $\widehat{B}$
and has finitely many orbits.
Then we may consider the orbit category $\widehat{B}/G$
defined as follows, see \cite{Ga}.
The objects of $\widehat{B}/G$ are the $G$-orbits
of objects of $\widehat{B}$
and the morphism spaces are given by
\[
   \big(\!\widehat{B}/G\big)(a,b) =
   \bigg\{
     f_{y,x} \in \prod_{(x,y) \in (a,b)\!\!\!\!\!\!\!\!\!\!\!\!\!\!\!\!} \widehat{B}(x,y)
     \ \Big|\ g f_{y,x} = f_{gy,gx} \mbox{ for all } g \in G, x \in a, y \in b
   \bigg\}
\]
for all objects $a,b$ of $\widehat{B}/G$.
Then $\widehat{B}/G$ is a bounded selfinjective $K$-category
which we identify with the associated finite dimensional
selfinjective $K$-algebra.

An automorphism $\varphi$ of the $K$-category $\widehat{B}$
is called:
\begin{itemize}
 \item
  \emph{positive} if, for every $(m,i)$, we have
  $\varphi(e_{m,i}) = e_{p,j}$
  for some $p \geqslant m$ and $j \in \{1,\dots,n\}$;
 \item
  \emph{rigid} if, for every $(m,i)$, we have
  $\varphi(e_{m,i}) = e_{m,j}$
  for some $j \in \{1,\dots,n\}$;
 \item
  \emph{strictly positive} if it is positive and not rigid.
\end{itemize}

Thus, for instance, the automorphisms $\nu_{\widehat{B}}^n$,
with $n \geqslant 1$, are strictly positive automorphisms
of $\widehat{B}$.

We recall that an algebra $B$ is called \emph{tilted}
if there exists a basic and connected hereditary $K$-algebra $H$
and a multiplicity-free tilting $H$-module $T$ such that $B = \End T$.
Moreover, $B$ is said to be of Dynkin, Euclidean or wild type according as
the valued quiver of $H$ is a Dynkin, Euclidean or wild quiver,
respectively.

We have the following general result, see \cite[Theorem~7.1]{SY4}.

\begin{proposition}
\label{prop:4.1}
Let $B$  be a tilted algebra and $G$ an admissible
torsion-free automorphism group of $\widehat{B}$.
Then $G$ is an infinite cyclic group generated by
a strictly positive automorphism $\varphi$ of $\widehat{B}$.
\end{proposition}

By \emph{selfinjective algebra of tilted type},
we mean an orbit algebra $\widehat{B}/G$, where
$B$ is a tilted algebra and $G$ is an admissible
infinite cyclic group of automorphisms of $\widehat{B}$.
Moreover, a selfinjective algebra $A = \widehat{B}/G$
of tilted type is said to be of Dynkin, Euclidean or wild type
according as the tilted algebra is of Dynkin, Euclidean or wild type,
respectively.

We note that $A$ is representation-finite if and only if
$B$ is of Dynkin type.

The following corollary is the first application of our
Theorem~\ref{th:A}
and the results of
\cite{AST1}, \cite{AST2}.

\begin{corollary}
\label{cor:4.2}
Let $A$ be a selfinjective algebra socle equivalent to
a representation infinite selfinjective algebra of tilted type.
Then $\repdim A = 3$.
\end{corollary}

\begin{proof}
Assume $A$ is socle equivalent to a representation-infinite
selfinjective algebra $A' = \widehat{B}/G$ of tilted type.
It follows from
\cite[Theorem]{AST1} and \cite[Theorem~A]{AST2}
that $\repdim A' = 3$ if the
ground
field $K$ is algebraically closed.
In fact, we gave an explicit construction of
an Auslander generator for $\mod A'$, applying the canonical
Galois covering functor $\widehat{B} \to \widehat{B}/G = A'$.
But the arguments used in \cite{AST1} and \cite{AST2}
remain valid for algebras over an arbitrary field $K$,
thanks to general results on selfinjective algebras
of tilted type, presented in \cite{SY4}.
\end{proof}

Let $A$ be a selfinjective algebra.
We denote by $\tau_A$ the Auslander-Reiten translation
in $\mod A$ and by $\Gamma(\mod A)$ its Auslander-Reiten
quiver.
A full valued subquiver $\Delta$ of $\Gamma(\mod A)$
is called a \emph{stable slice} \cite{SY6} if the following
conditions are satisfied:
\begin{enumerate}[(1)]
 \item
  $\Delta$ is connected, acyclic and without projective modules.
 \item
  For any valued arrow $V \xrightarrow{(d,d')} U$ in $\Gamma(\mod A)$
  with $U$ in $\Delta$ and $V$ nonprojective,
  $V$ belongs to $\Delta$ or to $\tau_A \Delta$.
 \item
  For any valued arrow $V \xrightarrow{(d,d')} U$ in $\Gamma(\mod A)$
  with $V$ in $\Delta$ and $U$ noninjective,
  $U$ belongs to $\Delta$ or to $\tau_A^{-1} \Delta$.
\end{enumerate}

A stable slice $\Delta$ of $\Gamma(\mod A)$
is called \emph{regular} if $\Delta$ contains
neither the socle factor $P / \soc P$ nor the radical $\rad P$
of an indecomposable projective $A$-module $P$.
A stable slice $\Delta$ of $\Gamma(\mod A)$
is called \emph{$\tau_A$-rigid} if $\Hom_A(X, \tau_A Y) = 0$
for all indecomposable modules $X,Y$ from $\Delta$.
Because of a result proved in \cite{S1},
a  $\tau_A$-rigid stable slice $\Delta$ of $\Gamma(\mod A$)
is always finite.

Our Theorem~\ref{th:B} is the second application
of Theorem~\ref{th:A} and the main result of \cite{SY6}.

\begin{theorem}
\label{th:4.3}
Let $A$ be a representation-infinite selfinjective algebra
admitting a $\tau_A$-rigid stable slice in $\Gamma(\mod A)$.
Then $\repdim A = 3$.
\end{theorem}

\begin{proof}
Assume $\Delta$ is a $\tau_A$-rigid stable slice in $\Gamma(\mod A$).
Let $M$ be the direct sum of all the indecomposable $A$-modules
lying on $\Delta$, $I = \{ a \in A \,|\, M a = 0 \}$ the right annihilator of $M$
and $B = A/I$.
We claim that there exist, for some $r,s \geqslant 1$,
a monomorphism $\tau_A M \to M^r$ and
an epimorphism $M^{s} \to \tau_A^{-1} M$
in $\mod A$, and hence in $\mod B$.

Because $\Delta$ is a regular stable slice in $\Gamma(\mod A)$,
an injective envelope $f : \tau_A M \to I(\tau_A M)$ of $\tau_A M$
and a projective cover $g : P(\tau_A^{-1} M) \to \tau_A^{-1} M$
of $\tau_A^{-1} M$ in $\mod A$ factor through
$M^r$ and $M^s$, respectively, for some $r,s \geqslant 1$.
This establishes the claim.

In particular, $\Hom_A(M, \tau_A M) = 0$ implies that also
$\Hom_A(\tau_A^{-1} M, M) = 0$, so $\Delta$ is a double
$\tau_A$-rigid stable slice of $\Gamma(\mod A)$.
Then, because of \cite[Proposition~3.8]{SY6},
the following statements hold:
\begin{enumerate}[(a)]
 \item
  $M$ is a tilting $B$-module,
 \item
  $H = \End_B M$ is a hereditary algebra,
 \item
  $T = D(M)$ is a tilting $H$-module, and
 \item
  $B = \End_H T$.
\end{enumerate}
Applying  \cite[Theorem~2]{SY6}, we conclude that $A$ is socle equivalent
to an orbit algebra $A' = \widehat{B}/(\varphi \nu_{\widehat{B}})$
for some positive automorphism $\varphi$ of $\widehat{B}$.
Moreover, $B$ is not of Dynkin type.
Therefore $A'$ is a representation-infinite selfinjective algebra
of tilted type.
Applying now Corollary~\ref{cor:4.2}, we get that
$\repdim A = 3$.
We note that the algebras $A$ and $A'$ are not necessarily isomorphic
(see Example~\ref{ex:5.4}).
\end{proof}

Recall that a connected component $\cC$ of an Auslander-Reiten
quiver $\Gamma(\mod A)$ is called \emph{generalised standard} \cite{S2}
whenever, for two modules $X,Y$ in $\cC$, we have
$\rad_A^{\infty}(X,Y) = 0$.
Here, $\rad_A^{\infty}$ denotes the infinite radical of $\mod A$.

We have the following consequence of
Theorem~\ref{th:4.3}, extending \cite[Theorem~B]{AST2}
to algebras over an arbitrary field.

\begin{corollary}
\label{cor:4.4}
Let $A$ be a connected selfinjective algebra admitting
an acyclic generalised standard Auslander-Reiten component.
Then $\repdim A = 3$.
\end{corollary}

\begin{proof}
Let $\cC$ be an acyclic generalised standard component in
$\Gamma(\mod A)$.
Then $\cC$ is an infinite component admitting a $\tau_A$-rigid
regular stable slice, because it contains only finitely many
projective modules.
In particular, $A$ is representation-infinite.
Applying Theorem~\ref{th:B} yields $\repdim A = 3$.
\end{proof}

We refer to \cite{SY2}, \cite{SY3} for the structure of module
categories of selfinjective algebras admitting generalised standard
acyclic Auslander-Reiten components.

\section{Examples}
\label{sec:examples}

The aim of this section is to present illustrative examples.
The first two describe selfinjective algebras over an algebraically
closed field which are socle equivalent but not isomorphic
to selfinjective algebras of Euclidean and wild types.

\begin{example}
\label{ex:5.1}
Let $K$ be an algebraically closed field and $Q$ be the quiver
\[
    \xymatrix@C=3pc{ 1 \ar@(dl,ul)[]^{\alpha} \ar@<+.45ex>[r]^{\gamma} & 2 \ar@<+.45ex>[l]^{\beta}} .
\]
Consider the
quotient algebras
$A = K Q/I$ and $A' = K Q/I'$,
where $I$ and $I'$ are the ideals
\begin{align*}
  I &= (\alpha^2 - \alpha \gamma \beta,
      \alpha \gamma \beta + \gamma \beta \alpha,
      \beta \gamma,
      \beta \alpha \gamma \beta), \\
  I' &= (\alpha^2,
      \alpha \gamma \beta + \gamma \beta \alpha,
      \beta \gamma,
      \beta \alpha \gamma \beta).
\end{align*}
Then $A' \cong \widehat{B}/(\varphi)$,
where $B = K \Delta / J$ is the tilted algebra of Euclidean type
$\tilde{\bA}_3$ given by the quiver $\Delta$
\[
  \xymatrix@R=1pc@C=1.5pc{
      && 2 \ar[lld]_{\beta} \\
    1 &&&& 3  \ar[llu]_{\gamma} \ar[llll]_{\alpha}
       && 4  \ar[ll]_{\sigma}
  }
\]
and the ideal $J = (\sigma \gamma)$,
and $\varphi$ is a strictly positive automorphism of $\widehat{B}$
such that there exists a rigid automorphism $\rho$ with
$\varphi^2 = \rho \nu_{\widehat{B}}$.
Moreover, $A / \soc A$ and $A' / \soc A'$ are isomorphic
to the algebra $A^* = K Q/I^*$, where
$I^* = (\alpha^2, \alpha \gamma \beta, \gamma \beta \alpha, \beta \gamma)$.
Hence $A$ and $A'$ are socle equivalent, and therefore
$\repdim A = \repdim A' = 3$,
because of Theorem~\ref{th:B}.
On the other hand, it is easily seen that $A$ and $A'$
are not isomorphic.
We refer to \cite{BoS1} for a general construction of such
socle equivalent algebras.
We also note that $A$ and $A'$ are not stably equivalent,
see \cite[Theorem~1.2]{BoS2}.
\end{example}

\begin{example}
\label{ex:5.2}
Let $K$ be an algebraically closed field and $Q$ be the quiver
\[
    \xymatrix@C=3pc{
      3 \ar@<+.45ex>[r]^{\eta} & \ar@<+.45ex>[l]^{\delta}
      1 \ar@(ul,ur)[]^{\alpha} \ar@<+.45ex>[r]^{\gamma} & 2 \ar@<+.45ex>[l]^{\beta}
    } .
\]
Consider the
quotient algebras
$A = K Q/I$ and $A' = K Q/I'$,
where $I$ and $I'$ are the ideals
\begin{align*}
  I &= (\alpha^2 - \alpha \gamma \beta,
      \alpha \gamma \beta + \gamma \beta \alpha,
      \beta \gamma,
      \beta \alpha \gamma \beta,
      \alpha \gamma \beta - \delta \eta,
      \beta \delta,
      \eta \alpha,
      \eta \gamma,
      \alpha \delta), \mbox{ and} \\
  I' &= (\alpha^2,
      \alpha \gamma \beta + \gamma \beta \alpha,
      \beta \gamma,
      \beta \alpha \gamma \beta,
      \alpha \gamma \beta - \delta \eta,
      \beta \delta,
      \eta \alpha,
      \eta \gamma,
      \alpha \delta).
\end{align*}
Let $C = K \Delta / J$ be the quotient algebra
of the path algebra $K \Delta$ of the quiver $\Delta$
\[
  \xymatrix@R=1pc@C=1.5pc{
      && 2 \ar[lld]_{\beta} \\
    1 &&&& 3  \ar[llu]_{\gamma} \ar[llll]_{\alpha} \ar[ld]^{\delta}
       && 4  \ar[ll]^{\sigma} \\
    & 5 \ar[lu]^{\eta} && 6
  }
\]
by the ideal $J = (\sigma \gamma, \sigma \delta)$.
Then the Auslander-Reiten quiver $\Gamma(\mod C)$ of $C$
admits a unique preinjective component having a section
of the form
\[
  \xymatrix@R=1pc{
    && I_6 \ar[rd] \\
    & I_2 \ar[rr] && S_3 \ar[rd] \\
    I_1 \ar[ru] \ar[rrrr] \ar[rd] &&&& I_3 \\
    & I_5 & .
  }
\]
It follows from \cite[Theorem~5.6]{ASS} that $C$
is a tilted algebra of wild type
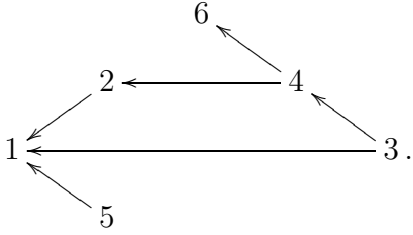
\[
  \xymatrix@R=1pc{
    && 6 \ar@{<-}[rd] \\
    & 2 \ar@{<-}[rr] && 4 \ar@{<-}[rd] \\
    1 \ar@{<-}[ru] \ar@{<-}[rrrr] \ar@{<-}[rd] &&&& 3 \,.\!\!\! \\
    & 5
  }
\]
Moreover, a simple checking shows that $A'$ is isomorphic to
the orbit algebra $\widehat{C}/(\psi)$
where $\psi$ is a strictly positive automorphism of $\widehat{C}$
such that there exists a rigid automorphism $\rho$
of $\widehat{C}$ with $\psi^2 = \rho \nu_{\widehat{C}}$.
Hence $A'$ is a selfinjective algebra of wild tilted type.
Further, $A / \soc A$ and $A' / \soc A'$ are both isomorphic
to the quotient algebra $A^* = K Q/I^*$, where
\[
  I^* = (\alpha^2,
      \alpha \gamma \beta,
      \gamma \beta \alpha,
      \beta \gamma,
      \delta \eta,
      \beta \delta,
      \eta \alpha,
      \eta \gamma,
      \alpha \delta) .
  \]
Therefore, $A$ and $A'$ are socle equivalent,
while they are clearly not isomorphic.
Applying Theorem~\ref{th:B}, we obtain
$\repdim A = \repdim A' = 3$.
\end{example}

\begin{example}
\label{ex:5.3}
Let $K$ be an algebraically closed field.
To each nonzero element $\lambda \in K$,
we associate the four-dimensional local selfinjective algebra
\[
   A(\lambda) = K \langle x, y \rangle / (x^2, y^2, x y - \lambda y x) .
\]
For any nonzero elements $\lambda, \mu$,
the algebras $A(\lambda)$ and $A(\mu)$
are socle  equivalent.
On the other hand, it was shown by Rickard
that the algebras $A(\lambda)$ and $A(\mu)$
are stably  equivalent if and only if
$\mu = \lambda$ or $\mu = \lambda^{-1}$,
in which case $A(\lambda)$ and $A(\mu)$
are also isomorphic.
This was done by a careful analysis of actions of the
syzygy operator on the indecomposable $2$-dimensional
modules forming the mouth of the stable tubes of rank $1$
in the stable Auslander-Reiten quiver of the algebra
$A(\lambda)$, see \cite[Example~IV.10.7]{SY5}
for a description of these actions.

We note that $A(1)$ is a selfinjective algebra of the from
$A(1) = \widehat{H}/(\varphi)$, where $H$ is the path algebra
of the Kronecker quiver
\[
    \xymatrix@C=3pc{
      1 \ar@<+.45ex>[r]^{\alpha}\ar@<-.45ex>[r]_{\beta} & 2
    }
\]
and $\varphi$ is an automorphism of $\widehat{H}$ with
$\varphi^2 = \nu_{\widehat{H}}$.
Because of \cite[Theorem]{AST1},
we have $\repdim A(1) = 3$.
Applying now our Theorem~\ref{th:B}, we get that
$\repdim A(\lambda) = 3$ for any $\lambda \in K \setminus \{0\}$.

We also note that $A(-1)$ is the exterior algebra $\Lambda(K^2)$.
It follows from \cite[Theorem~4.1]{R} that,
for any integer $n \geqslant 2$, the exterior algebra $\Lambda(K^n)$
of the $n$-dimensional vector space $K^n$ has representation
dimension $n+1$.
It is thus very natural to expect that there are many selfinjective
algebras $A$ socle equivalent but not isomorphic
(even not stably equivalent) to the exterior algebra
$\Lambda(K^n)$, and then such that
$\repdim A = n+1$.
\end{example}

The next example shows that socle equivalences exist
naturally for Hochschild extensions of hereditary algebras
by duality bimodules.
We refer to
\cite[Chapter~X]{SY7}
for the general theory of Hochschild extensions.

\begin{example}
\label{ex:5.4}
Let $K$ be a field of characteristic $2$,
and $L$ be a finite field extension of $K$ such that the
Hochschild cohomology group $H^2(L,L)$, where
$L$ is considered as a $K$-algebra, is nonzero.
We refer to
\cite[Section~X.5]{SY7}
for such field extensions.

Take a $2$-cocycle $\alpha : L \times L \to L$
corresponding to a nonsplit extension
$0 \to L \to M \to L \to 0$.
For example, we may take $K = \bZ_2(u)$,
the field of rational functions in one variable $u$
over the field $\bZ_2 = \bZ / 2 \bZ$,
and $L = K[X]/(X^2-u)$,
where $K[X]$ is the polynomial algebra
in one variable $X$ over $K$.
Denoting by $x$ the residual class $X+(X^2 - u)$,
we see that $L$ has $\{1,x\}$ as a $K$-basis, and
a nonsplit $2$-cocycle $\alpha : L \times L \to L$ as required
above is given by
\[
  \alpha(x^l, x^m) = x^{l+m}
\]
for $l,m \in \{0,1\}$, see \cite[Example~X.5.4]{SY7}.

Let $Q = (Q_0, Q_1)$ be a finite connected acyclic quiver
without double arrows, and $H = L Q$ be its path algebra
over $L$.
For each point $i \in Q_0$, choose a primitive idempotent $e_i$
of $H$ and for each path from $i$ to $j$, choose an element
$h_{j i} = e_j h_{j i} e_i$ of $H$.
Then $D H = \Hom_L(L Q, L) \cong \Hom_K(L Q, K)$
has a dual basis $e^*_i, h^*_{j i}$ over $L$.
Let $\widetilde{H} = H \oplus D H$ be the direct sum of $H$
and $D H$ considered as $K$-spaces, and define a multiplication
on $\widetilde{H}$ in the following way
\[
  (a,v)(b,w)
   = \Big(a b, a w + v b + \sum_{i \in Q_0} \alpha(a_i,b_i) e^*_i \Big)
\]
for $a,b \in H$, $v,w \in D H$, where $a_i,b_i \in L$ are such that
\begin{align*}
  a&= \sum a_i e_i + \sum r_{j i} h_{j i},
  &
  b&= \sum b_i e_i + \sum s_{j i} h_{j i},
\end{align*}
for $r_{j i}, s_{j i} \in L$ are the basis presentations of $a$ and $b$.

Letting $\rho : \widetilde{H} \to H$ denote the canonical epimorhism
and $\omega : D H \to \widetilde{H}$ the embedding,
we have a nonsplit Hochschild extension
\[
  0 \to D H \xrightarrow{\omega} \widetilde{H}  \xrightarrow{\rho} H \to 0,
\]
see \cite[Theorem~X.6.7]{SY7}.
Moreover, $\widetilde{H}$ is selfinjective, and even
weakly symmetric, and the elements
$\widetilde{e}_i = (e_i, - \alpha(1,1) e^*_i) \in \widetilde{H}$
form a complete set of orthogonal primitive idempotents
of $\widetilde{H}$.
Because of \cite[Corollary~4.2]{SY1},
$\widetilde{H}$ is socle equivalent to the trivial extension
$\operatorname{T}(H) = H \ltimes D H$.
We also know that $\widetilde{H}$ is not isomorphic to
an orbit algebra $\widehat{B}/(\varphi \nu_{\widehat{B}})$
where $B$ is a $K$-algebra and $\varphi$ is a positive
automorphism of $\widehat{B}$ (see \cite[Proposition~4]{SY3}).
Clearly,
$\widetilde{H}$ and $\operatorname{T}(H)$ are not isomorphic.
\end{example}

We end with an example of socle equivalence of symmetric
algebras arising from triangulated surfaces.

\begin{example}
\label{ex:5.5}
Let $K$ be an algebraically closed field,
$m$ a positive natural number, $c$ a nonzero scalar
from $K$,
and $b_{\bullet} : \{1,2,3\} \to K$ a function.

Consider the surface $S$ of the triangle $T$
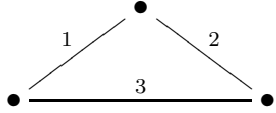
\[
  \xymatrix@R=2pc@C=3pc{
      & \bullet \ar@{-}[ld]_{1} \\
    \bullet && \bullet  \ar@{-}[lu]_{2} \ar@{-}[ll]_{3}
  }
\]
where $1,2,3$ are boundary edges, and let $\vec{T}$
be the clockwise orientation $(1\ 2 \ 3)$ of the edges of $T$.

According to \cite[Section~4]{ES1},
we may associate to the pair $(S,\vec{T})$
the triangulation quiver $(Q(S,\vec{T}),f)$ of the form
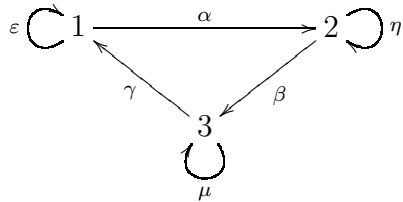
\[
  \xymatrix@R=2pc@C=3pc{
     1 \ar@(dl,ul)[]^{\varepsilon} \ar[rr]^{\alpha} &&
     2 \ar@(ur,dr)[]^{\eta} \ar[ld]^{\beta} \\
     & 3 \ar@(dr,dl)[]^{\mu} \ar[lu]^{\gamma}}
\]
where $f$ is the permutation of arrows defined as follows
\begin{align*}
 f(\alpha) &= \beta, &
 f(\beta) &= \gamma, &
 f(\gamma) &= \alpha, &
 f(\varepsilon) &= \varepsilon, &
 f(\eta) &= \eta, &
 f(\mu) &= \mu
\end{align*}
(see \cite[Example~4.3]{ES1}).
Consider the bound quiver algebra
\[
  \Lambda(S,\vec{T},m,c,b_{\bullet})
  = K Q\big(S,\vec{T})/ I(Q(S,\vec{T}),f,m,c,b_{\bullet}\big)
\]
where $I(Q(S,\vec{T}),f,m,c,b_{\bullet})$
is the admissible ideal in the path algebra $K Q(S,\vec{T})$
generated by the elements
\begin{align*}
 &\alpha\beta - c(\varepsilon\alpha\eta\beta\mu\gamma)^{m-1} \varepsilon\alpha\eta\beta\mu ,
 &
 \varepsilon^2 &- c(\alpha\eta\beta\mu\gamma\varepsilon)^{m-1} \alpha\eta\beta\mu\gamma
- b_1(\alpha\eta\beta\mu\gamma\varepsilon)^{m},
\\
 &\beta\gamma - c(\eta\beta\mu\gamma\varepsilon\alpha)^{m-1} \eta\beta\mu\gamma\varepsilon ,
 &
 \eta^2 &- c(\beta\mu\gamma\varepsilon\alpha\eta)^{m-1} \beta\mu\gamma\varepsilon\alpha
- b_2(\beta\mu\gamma\varepsilon\alpha\eta)^{m},
\\
 &\gamma\alpha - c(\mu\gamma\varepsilon\alpha\eta\beta)^{m-1} \mu\gamma\varepsilon\alpha\eta,
 &
 \mu^2 &- c(\gamma\varepsilon\alpha\eta\beta\mu)^{m-1} \gamma\varepsilon\alpha\eta\beta
- b_3(\gamma\varepsilon\alpha\eta\beta\mu)^{m},
\\
 &
 \quad
\alpha\beta\mu,
 \qquad
 \ \quad
 \varepsilon^2 \alpha,
 \qquad
 \ \quad
 \beta\gamma\varepsilon,\!\!\!\!\!\!\!\!\!\!\!\!\!\!
 &&
 \quad
 \ \,\qquad
 \eta^2 \beta,
 \qquad
 \ \quad
 \gamma\alpha\eta,
 \ \quad
 \qquad
 \mu^2 \gamma.
\end{align*}
We also denote by $0$ the zero function from $ \{1,2,3\}$ to $K$
and set
$\Lambda(S,\vec{T},m,c) = \Lambda(S,\vec{T},m,c,0)$.
Following \cite{ES1},
$\Lambda(S,\vec{T},m,c)$ is called a \emph{weighted surface algebra}
of $(S,\vec{T})$.
Because of \cite[Propositions 8.1 and 8.2]{ES1},
$\Lambda(S,\vec{T},m,c)$ and $ \Lambda(S,\vec{T},m,c,b_{\bullet})$
are socle equivalent representation-infinite tame symmetric
algebras of dimension $36m$, which are isomorphic
if the characteristic of $K$ is different from $2$.
On the other hand, it was shown in \cite[Example~8.4]{ES1} that,
if $\operatorname{char} K = 2$, $m = 1$ and $b_{\bullet}$ is nonzero,
then the algebras
$\Lambda(S,\vec{T},m,c)$ and $ \Lambda(S,\vec{T},m,c,b_{\bullet})$
are not isomorphic.
We refer to \cite[Section~8]{ES1} and \cite[Section~6]{ES2}
for socle equivalence of representation-infinite tame symmetric
algebras associated to arbitrary triangulated surfaces with
nonempty boundary.
\end{example}

\section*{Acknowledgements}
All authors were supported by the research grant DEC-2011/02/A/ST1/
00216 of the National Science Center Poland.
The first author was also partially supported by the NSERC of Canada, and the third author was also partially supported by ANPCyT, Argentina.
The paper was completed during the visit of I. Assem and S. Trepode
at Nicolaus Copernicus University in Toru\'n (August 2017).

The results of the paper were partially presented by the second
named author during the conference
``Idun 75. A conference on representation theory of artin algebras
on the occasion of Idun Reiten's birthday'' (Trondheim, May 2017)
and by the third named author during the
``Joint Meeting of Sociedad Matem\'atica Espa\~nola and Uni\'on Matem\'atica Argentina''
(Buenos Aires, December 2017).

%
%








\section{References}
\bibliographystyle{amsalpha}

\end{document}